\documentclass[11pt]{amsart}

\usepackage{amsmath}
\usepackage{amssymb}
\usepackage{bbm}
\usepackage{pdfsync}
\usepackage{cleveref}
\usepackage{mathtools}
\usepackage{xcolor}
\usepackage{xparse}

\date{\today}
\allowdisplaybreaks
\newcommand{\R}{{\mathbb R}}       
\newcommand{\N}{{\mathbb N}}
\newcommand{\Sp}{{\mathbb S}}
\newcommand{\Z}{{\mathbb Z}}       
\newcommand{\Toh}{{ T_{\Omega, h}}}
\DeclarePairedDelimiter\abs{\lvert}{\rvert}
\DeclarePairedDelimiter\norm{\lVert}{\rVert}

\newcommand{\vertiii}[1]{{\left\vert\kern-0.25ex\left\vert\kern-0.25ex\left\vert #1 
		\right\vert\kern-0.25ex\right\vert\kern-0.25ex\right\vert}}

\DeclarePairedDelimiterX{\set}[1]{\{}{\}}{\setargs{#1}}
\NewDocumentCommand{\setargs}{>{\SplitArgument{1}{;}}m}
{\setargsaux#1}
\NewDocumentCommand{\setargsaux}{mm}
{\IfNoValueTF{#2}{#1} {#1\,\delimsize|\,\mathopen{}#2}}

\newtheorem{theorem}{Theorem}[section]
\newtheorem{lemma}[theorem]{Lemma}

\newtheorem*{theorem*}{Theorem}

\theoremstyle{definition}
\newtheorem{definition}[theorem]{Definition}

\numberwithin{equation}{section}

\begin{document}
	
	\title{ Weak type bounds for rough maximal singular integrals near $L^1$}

	\author{Ankit Bhojak}
	\address{Ankit Bhojak\\
		Department of Mathematics and Statistics\\
		Indian Institute of Technology Kanpur\\
		Kanpur-208016, India.}
	\email{ankitbjk@iitk.ac.in}
	
	\author{Parasar Mohanty}
	\address{Parasar Mohanty\\
		Department of Mathematics and Statistics\\
		Indian Institute of Technology Kanpur\\
		Kanpur-208016, India.}
	\email{parasar@iitk.ac.in}
	
	\thanks{}

	\begin{abstract}
		In this paper it is shown that for $\Omega\in L\log L(\Sp^{d-1})$, the rough maximal singular integral operator $T_\Omega^*$   is of weak type $L\log\log L (\R^d)$. Furthermore,  for $w\in A_1$ and $\Omega\in L^\infty(\Sp^{d-1})$,  it is  shown that  $T_\Omega^*$  is of weak type   $L\log\log L (w)$  with weight dependence  $[w]_{A_1}[w]_{A_{\infty}}\log([w]_{A_{\infty}}+1),$ which is same as the best known constant for the singular integral $T_\Omega$.
	\end{abstract}

	\maketitle
	
	\section{Introduction}
	The rough maximal operator $M_{\Omega}$ for  $\Omega\in L^1(\Sp^{d-1})$  is defined as
	\[M_{\Omega}f(x)=\sup\limits_{r>0}\;\frac{1}{r^d}\int\limits_{\abs{x-y}\leq r}\abs[\Big]{\Omega\Bigl(\frac{x-y}{\abs{x-y}}\Bigr)f(y)}\;dy.\]
	Let   $\int_{\Sp^{d-1}}\Omega(\theta) d\theta=0$, where $d\theta$ is the surface measure on $\Sp^{d-1}$.  The rough singular integral is given by,
	\[ T_{\Omega}f(x)=p.v.\int\frac{1}{|x-y|^d}\Omega\Bigl(\frac{x-y}{|x-y|}\Big)f(y)\;dy,\]
	and the truncated  rough singular operator is defined as 
	\[ T^\epsilon_\Omega f(x) = \int\limits_{|x-y|>\epsilon} \frac{1}{|x-y|^d}\Omega\Bigl(\frac{x-y}{|x-y|}\Big)f(y)\;dy.\]
	Consider the  rough maximal singular operator,  
	\[ T_{\Omega}^*f(x)= \sup\limits_{\epsilon>0} |T^\epsilon_\Omega f(x)|.\]

	For  $\Omega\in L \log L(\Sp^{d-1})$  the rough singular integrals $T_\Omega$ as well as rough maximal singular integral operator  $T_\Omega^*$,  were shown to be bounded  in $L^p(\R^d)$ for $1<p<\infty$ by Calder\'{o}n and Zygmund \cite{CZ}. The same result was also established by Duoandikoetxea and Rubio de Francia \cite{DR}  for $\Omega\in L^q(\Sp^{d-1}),\; q>1$, using Fourier transform estimates and a double dyadic decomposition of the kernel $K(x)= p.v.\abs{x}^{-d}\Omega\bigl(\frac{x}{\abs{x}}\bigr)$.\\
	\par
	The case $p=1$  was more elusive. 
	For $\Omega\in {\rm Lip}(\Sp^{d-1})$,  $T_{\Omega}$ is a standard Calder\'{o}n-Zygmund operator hence bounded from $L^1(\R^d)$ to $L^{1,\infty}(\R^d)$.   In \cite{C} and \cite{CR}, Christ and Rubio de Francia showed that $M_\Omega$ is weak (1,1) for $\Omega\in L\log L(\Sp^{d-1})$. It was shown that $T_\Omega$ is weak $(1,1)$ in dimension two independently in \cite{CR} and \cite{Hof}. Finally, the case  for all dimensions  for $T_\Omega$ was resolved by Seeger \cite{S1} using microlocal analysis. However, the weak $(1,1)$ boundedness for $T_{\Omega}^*$ is an open problem even for $\Omega\in L^\infty(\Sp^{d-1})$, see \cite{S2}.\\
	\par
	The aim of this paper is to study the end point boundedness properties of the operator $T_{\Omega}^*$. 
	Using an extrapolation argument one can derive that $T_\Omega^* f$ is in $L^{1,\infty}$ if $f\in L\log L$ and supported in a cube 
	By a slight abuse of notation, we define an operator $T$ is of weak type $L\log\log L(\mu)$ for a measure $\mu$ if there exists a constant $C_{\mu,T}>0$ such that
	\[\mu(x\in\R^d:\abs{Tf(x)}>\alpha)\leq \frac{C_{\mu,T}}{\alpha}\int \abs{f(x)}\log\log\left(e^2+\frac{\abs{f(x)}}{\alpha}\right)\;d\mu (x).\]
	holds for all $\alpha>0$. We denote $\vertiii{T}_{L\log\log L(\mu)\to L^{1,\infty}(\mu)}$ to be the smallest constant $C_{\mu,T}>0$ satisfying the above inequality. We also write $T$ is weak type $L\log\log L$ if the underlying measure is the Lebesgue measure on $\R^d$. 
	
	Our main theorem is as follows:
	\begin{theorem}\label{Main}
		Let $\Omega\in L\log L(\Sp^{d-1})$ and $\int_{\Sp^{d-1}}\Omega(\theta) d\theta=0$. Then $T_{\Omega}^*$ is of weak type $L\log\log L$.
	\end{theorem}

	
	Our result improves the following result of Honz\'{i}k \cite{H},   proved very recently.  
	\begin{theorem}[\cite{H}]\label{Honzik}
		Let  $\Omega\in L^{\infty}(\Sp^{d-1})$ and  $\epsilon>0$. Then 
		\[ \|T_{\Omega}^*f\|_{L^{1,\infty}}\leq C_\epsilon\|f\|_{ L(\log\log L)^{2+\epsilon}} \]
		for functions $f$ supported in an unit cube and \[\norm{f}_{L(\log\log L)^{2+\epsilon}}=\inf\set[\Big]{\alpha>0:\int \abs{f}\left(\log\log \Bigl(e^2+\frac{\abs{f}}{\alpha}\Bigr)\right)^{2+\epsilon}\leq {\alpha}}.\]
	\end{theorem}
	We would like to indicate the following improvements of \Cref{Honzik}. 
	First it encompasses larger domain   $L\log\log L(\R^d)$, secondly  it   extends the class of $\Omega\in L^\infty(\Sp^{d-1})$ to $L\log L(\Sp^{d-1})$. Lastly,  
	\Cref{Main} also implies a local Orlicz space estimate like \Cref{Honzik}. Namely for $f$ supported in a unit cube $\mathcal{Q}$, we have
	\[ \|T_{\Omega}^*f\|_{L^{1,\infty}}\lesssim\norm{f}_{ L\log\log L(\mathcal{Q})}.\]
	To see this, assume $\norm{f}_{ L\log\log L(\mathcal{Q})}\leq 1$. For $\alpha\leq 100$, observe that $T_\Omega^*f(x)\lesssim M_\Omega f(x)$ holds for $x\notin((100d)\mathcal{Q})$ and hence the above inequality follows from weak $(1,1)$ boundedness of $M_\Omega$. The case $\alpha>100$ follows from \Cref{Main}.

	
	One can observe that applying the method  given in \cite{H} to $T_\Omega$ only yields  \linebreak $T_\Omega : L(\log\log L)^{1+\epsilon}(\R^d)\to L^{1,\infty}(\R^d)$ boundedly.   In the aforementioned  method, the function $f$ is decomposed based on the $L(\log\log L)^\gamma$ averages for $\gamma>0$ instead of the more natural $L^1$ averages. Also, the approach in \cite{H} relies on the double dyadic decomposition of the kernel $K$ at the beginning, as in \cite{HRT}, which forces extra assumption on the  size of the function.  
	In contrast  to  the  decomposition considered  in \cite{H}, our method is similar in the spirit of  \cite{S1}, in the sense that if we apply this method for $T_\Omega$ we will recover the well known weak $(1,1)$ estimate. We will employ  a finer decomposition of  the function based on its size  as in \cite{STW}.

	In this paper we have also studied the operator with rough radial part.  Let $h\in L^\infty(\R^n)$ be a radial function. Define 
	\[T_{\Omega,h} f(x)=p.v. \int_{\R^d} \frac{h(x-y)}{|x-y|^{-d}} \Omega\Bigl(\frac{x-y}{|x-y|}\Big)  f(y) dy.\]
	This operator was shown to be bounded from $L^p(\R^d)$ to $L^p(\R^d),\,1<p<\infty,$ for $\Omega\in {\rm Lip}(\Sp^{d-1})$ by R. Fefferman \cite{F}. Later  Duoandikoetxea and Rubio de Francia \cite{DR} improved this result for $\Omega\in L^q(\Sp^{d-1})$ for some $1<q\leq \infty$ and further to $\Omega\in H^1(\Sp^{d-1})$ in \cite{FP}. We would like to mention that Bochner-Riesz operator at the critical index $\frac{d-1}{2}$ is a prime example of this class of operators.
	In \cite{CR} weak $(1,1)$ estimate for this operator was shown, assuming the $L^2$ boundedness of $T_{\Omega,h}$ for  $\Omega\in L^\infty(\Sp^{d-1})$ and $\partial_\theta\Omega\in L^\infty(\Sp^{d-1})$.  By standard argument the size condition $\Omega\in L^\infty(\Sp^{d-1})$ can be replaced by $\Omega\in L\log L(\Sp^{d-1})$. To the best of our knowledge weak $(1,1)$ boundedness of $\Toh$ for $h\in L^\infty$ and for  $\Omega\in L\log L(\Sp^{d-1})$, without assuming any smoothness condition, is not known.  	
	
	Define the maximal operator corresponding to $\Toh$ as 
	\[T^*_{\Omega,h}f(x)=\sup\limits_{\epsilon>0} \;\abs[\Big]{ \int_{|x-y|>\epsilon} \frac{h(x-y)}{|x-y|^{-d}} \Omega\Bigl(\frac{x-y}{|x-y|}\Big)  f(y) dy}.\]
	In \cite{DR} it was proved that $T^*_{\Omega,h}$ is bounded from $L^p(\R^d)$ to $L^p(\R^d)$ for $1<p<\infty$.  Our result is the following
	\begin{theorem}\label{homega}
		Let $\Omega\in L\log L(\Sp^{d-1})$ with $\int_{\Sp^{d-1}} \Omega(\theta d\theta=0$, $\partial_\theta \Omega\in L^\infty(\Sp^{d-1})$, and $h\in L^\infty(\R^d)$ be a radial function.  Then $T^*_{\Omega,h}$ is of weak type $L\log\log L$.
	\end{theorem}
	
	In the last decade quantitative weighted boundedness of singular integral operator has been one of the main theme of research in harmonic analysis. The celebrated result of Hyt\"onen \cite{Hy} finally showed that the dependence of  the $L^2(w)$ boundedness of the  Calder\'on-Zygumd operator $T_{_{CZ}}$ on the $A_2$ characteristic   of the weight $w$ is linear.  The sparse techniques, which evolved,  while addressing this problem,  proved to be very significant.  It is still unknown whether the quadratic dependence of $T_\Omega: L^2(w) \to L^2(w)$ on the $A_2$ characteristic of $w$  is sharp  \cite{HRT, CCDO, L2}.   For the end point $p=1$, the operator $T_\Omega$ is known to be bounded from $L^1(w)\to L^{1,\infty}(w)$ for $w\in A_1$ (\cite{V, FS2}.  To state a quantitative  version of this result we will need the following  definitions. 
	\begin{definition}
		Let $w:\R^d\to\R$ be a non-negative function. We say $w\in A_1$ if
		\[[w]_{A_1}=\sup\limits_{Q\subset\R^d}\;\Bigl(\frac{1}{\abs{Q}}\int\limits_Qw(t)\;dt\Bigr)\norm{w^{-1}}_{L^{\infty}(Q)}\]
		is finite. For $1<p<\infty$, we say $w\in A_p$ if
		\[[w]_{A_p}=\sup\limits_{Q\subset\R^d}\;\Bigl(\frac{1}{\abs{Q}}\int\limits_Qw(t)\;dt\Bigr)\Bigl(\frac{1}{\abs{Q}}\int\limits_Qw(t)^{-\frac{1}{p-1}}\;dt\Bigr)^{p-1}\]
		is finite, and $w\in A_{\infty}$ if
		\[[w]_{A_{\infty}}=\sup\limits_{Q\subset\R^d}\;\Bigl(\int\limits_Qw(t)\;dt\Bigr)^{-1}\Bigl(\int\limits_{Q}M(w\chi_Q)(t)\;dt\Bigr)\]
		is finite, where $M$ is the Hardy-Littlewood maximal function.
	\end{definition}
	In \cite{LOP, HP1} it was shown that 
	\begin{equation}
		\norm{T_{_{CZ}}}_{L^1(w)\to L^{1,\infty}(w)}\lesssim [w]_{A_1}\log([w]_{A_\infty}+1).
	\end{equation}
	Moreover, this dependence on weight characteristics is optimal \cite{LNO}. In \cite{L1, HP2} it was shown that the above dependence also holds for the maximal Calder\'on-Zygmund operator $T_{_{CZ}}^*$. i.e.
	\begin{equation}
		\norm{T_{_{CZ}}^*}_{L^1(w)\to L^{1,\infty}(w)}\lesssim [w]_{A_1}\log([w]_{A_\infty}+1).
		\label{CZw}
	\end{equation}
	
	For $\Omega\in L^\infty(\Sp^{d-1})$  the following bound for $T_\Omega$  was obtained in   \cite{LPRR},
	\begin{equation}
		\norm{T_{\Omega}}_{L^1(w)\to L^{1,\infty}(w)}\lesssim [w]_{A_1}[w]_{A_\infty}\log([w]_{A_\infty}+1).
	\end{equation}
	It is not known whether the extra constant $[w]_{A_\infty}$, in comparison to $T_{_{CZ}}$,  can be removed.  The following result asserts a similar weighted dependence for $T_\Omega^*$. 
	\begin{theorem}\label{Mainw}
		Let $\Omega\in L^{\infty}(\Sp^{d-1})$ with $\int_{\Sp^{d-1}}\Omega(\theta)\;d\theta=0$. Then for $w\in A_1$ and $\alpha>0$, we have,
		\begin{equation}
			\vertiii{T_{\Omega}^*}_{L\log\log L(w)\to L^{1,\infty}(w)}\lesssim [w]_{A_1}[w]_{A_{\infty}}\log([w]_{A_{\infty}}+1).
		\end{equation}
	\end{theorem}
	We would like to remark that if one follows the method given in \cite{V} one gets that $T_\Omega^*$ is of weak type $L\log\log L(w)$. However our proof  is inspired by the method given in \cite{LPRR} which yields  better weight constant.  Following the arguments of \Cref{Mainw}, a similar result can be obtained for the operator $T^*_{\Omega,h}$.

	In section \S 2 we will give the proof of \Cref{Main},  \S 3 will contain the sketch of the proof of \Cref{homega}, and \S 4 is devoted to the proof of \Cref{Mainw}. Throughout this paper $A\lesssim B$ means there is a constant $C$ depending only on the dimension $d$ such that $A\leq CB$.  For any cube $Q$ we denote $l(Q)$ to be the length of the cube.  For us, a dyadic cube $Q$ means $Q=2^j\Bigl((0,1]^d+l\Bigr)$ for some $j\in\Z$ and $l\in\Z^d$. For $C>0$, we define $CQ$ to be the cube with the same centre as $Q$ and length $Cl(Q)$. Also, for a non-negative weight $w$ and a set $E\subset\R^d$, we denote $\abs{E}$ to be the Lebesgue measure of the set $E$ and $w(E)=\int_E\;w(t)\;dt$.
	\section{Proof of \Cref{Main}}
	We begin by discretising the supremum in $T_{\Omega}^*$. Let $\beta\in C_c^{\infty}(\R^d)$ be a function supported on the annulus $\{\frac{1}{2}\leq\abs{x}\leq 2\}$ and $\sum\limits_{i\in\Z}\beta_i(x)=1$ for $x\neq 0$, where $\beta_i(x)=\beta(2^{-i}x)$. We have
	\begin{equation}
		T_{\Omega}^*f\leq M_{\Omega}f+\sup\limits_{k\in\Z}\;\abs[\Big]{\sum\limits_{i>k}K^i*f},
		\label{discr}
	\end{equation}
	where $K^i(x)=K(x)\beta_i(x)$. To see the last inequality, for an $\epsilon>0$, let $k_{\epsilon}=[\log_2\epsilon]$. Then we have $T^\epsilon_{\Omega}f\leq \abs{K^{k_{\epsilon}}}*\abs{f}+\abs{\sum\limits_{i>k_{\epsilon}}K^i*f}\lesssim M_{\Omega}f+\sup\limits_{k\in\Z}\;\abs{\sum\limits_{i>k}K^i*f}$ and hence \eqref{discr} follows.

	$M_{\Omega}$ is known to be weak type $(1,1)$, see \cite{CR}. So it remains to estimate $\sup\limits_{k\in\Z}\;\abs{\sum\limits_{i>k}K^i*f}$.
	
	\subsection{Decomposition of the function} \label{function}
	Let $f\in L\log\log L(\R^d)$ and $\alpha>0$. We choose a maximal collection $\{Q\}_{Q\in\mathcal{F}}$ of dyadic cubes such that,
	\[\alpha<\frac{1}{\abs{Q}}\int\limits_{Q}\;\abs{f}\leq 2^d\alpha.\]
	We define the exceptional set $E=\cup_{_{Q\in\mathcal{F}}}(100d)Q$. Therefore $\abs{E}\lesssim\frac{1}{\alpha}\norm{f}_1$. 		We set $$f=g+\sum\limits_{Q\in\mathcal{F}}f_{_Q},$$ where 
	\begin{eqnarray*}
		g&=&f\chi_{_{\R^d\setminus \cup Q}}+f\chi_{_{(\cup Q)\cap\{\abs{f}\leq2^{c_1}\alpha\}}},\\
		f_{_{Q}} &=&f\chi_{_{Q\cap\{\abs{f}>2^{c_1}\alpha\}}},
	\end{eqnarray*}
	and $0<c_1<\frac{1}{4}$ is independent of $\alpha$ and is to be chosen later.\\
	By Lebesgue differentiation theorem, we have $\abs{g}\lesssim\alpha$ and hence $\norm{g}_2^2\lesssim\alpha\norm{f}_1$. We further decompose \[f_{_Q}=\sum\limits_{n=1}^{\infty}f_{_Q}^n,\] where
	\[f_{_Q}^n=f_{_Q}\chi_{\{2^{c_1 2^{(n-1)}}\alpha<\abs{f_{_Q}}\leq 2^{c_1 2^n}\alpha\}}.\]
	Clearly, we have $\sum\limits_{n=1}^{\infty}\frac{1}{\abs{Q}}\int\abs{f_{_Q}^n}\lesssim\alpha$. 		We write $f_{_Q}^n=g_{_Q}^n+b_{_Q}^n$, where
	\[g_{_Q}^n(x)=\Bigl(\frac{1}{\abs{Q}}\int f_{_Q}^n\Bigr)\chi_Q(x),\text{ and }b_{_Q}^n=f_{_Q}^n-g_{_Q}^n.\]
	For $n\in\N$, we set $g^n=\sum\limits_{Q\in\mathcal{F}}g_{_Q}^n, \;b^n=\sum\limits_{Q\in\mathcal{F}}b_{_Q}^n$, and $f^n=\sum\limits_{Q\in\mathcal{F}}f_{_Q}^n$.

	We state some basic properties of these functions which can be verified easily.
	\begin{align}
		\label{P1}\norm{f^n}_2^2&\lesssim 2^{c_12^n}\alpha\norm{f}_1.\\
		\label{P2}\int b_{_Q}^n&=0\;\forall n\in\N,\;Q\in\mathcal{F}.\\
		\label{P3}\norm[\Big]{\sum\limits_n\sum\limits_{Q\in\mathcal{F}}\abs{g_{_Q}^n}}_{\infty}&\lesssim\alpha.\\
		\label{P4}\sum\limits_n\sum\limits_{Q\in\mathcal{F}}\Bigl(\norm{g_{_Q}^n}_1+\norm{b_{_Q}^n}_1\Bigr)&\lesssim\sum\limits_n\sum\limits_{Q\in\mathcal{F}}\norm{f^n_{_Q}}_1\lesssim\norm{f}_1.
	\end{align}
	By (\ref{P3}) and (\ref{P4}), we have $\norm{\sum\limits_ng^n}_2^2\lesssim\alpha\norm{f}_1$.
	Let $\phi\in C_c^{\infty}(\R^d)$ be a function supported in the ball $B(0,\frac{1}{2})$ satisfying $\int\phi=1$.
	We smoothen the kernel $K^i$ by writing,
	\begin{align*}
		K_0^i&=K^i*\phi_{-i},\\
		K_n^i&=K^i*\phi_{2^n-i}\;\text{for}\;n\in\N,
	\end{align*}
	where $\phi_j(x)=2^{jd}\phi(2^jx)$.\\
	To estimate the level set $\{x\in E^c:\sup\limits_{k\in\Z}\;\abs{\sum\limits_{i>k}K^i*f}>\alpha\}$, We write
	\begin{equation}
		\begin{aligned}\label{break}
			\sum\limits_{i>k}K^i*f=&\sum\limits_{i>k}K^i*g+\sum\limits_{i>k}\sum\limits_{n\geq 1}(K^i-K_n^i)*f^n+\sum\limits_{i>k}\sum\limits_{n\geq 1}(K_n^i-K_0^i)*g^n\\
			&+\sum\limits_{i>k}\sum\limits_{n\geq 1}K_0^i*f^n+\sum\limits_{i>k}\sum\limits_{n\geq 1}(K_n^i-K_0^i)*b^n.
		\end{aligned}
	\end{equation}
	We define 
	\[\mathcal{H}_{k,1}= \sum\limits_{i>k}K^i*g,\hspace{0.7cm}\mathcal{H}_{k,2} = \sum\limits_{i>k}\sum\limits_{n\geq 1}(K^i-K_n^i)*f^n,\hspace{1cm}\mathcal{H}_{k,3}=\sum\limits_{i>k}\sum\limits_{n\geq 1}(K_n^i-K_0^i)*g^n\]
	\[\mathcal{H}_{k,4}=\sum\limits_{i>k}\sum\limits_{n\geq 1}K_0^i*f^n,\hspace{0.7cm}\mathcal{H}_{k,b}= \sum\limits_{i>k}\sum\limits_{n\geq 1}(K_n^i-K_0^i)*b^n.\]
	\subsection{Estimate for Good parts}
	One can  observe that 
	\begin{equation}
		\abs{\widehat{K^i}(\xi)}\lesssim \min\{\abs{2^i\xi}^{a},\abs{2^i\xi}^{-a}\},\;0<a<1
		\label{Fourierestimate}
	\end{equation}
	also holds for $\Omega\in L\log L(\Sp^{d-1})$ instead of more restrictive  $\Omega\in L^{\infty}(\Sp^{d-1})$. 
	Taking into account of above observation and imitating the proof as in   \cite{DR} (Theorem E) we have the following. 
	\begin{lemma}
		Let $\Omega\in L\log L(\Sp^{d-1})$ with $\int_{\Sp^{d-1}} \Omega(\theta)d\theta=0$  and $g\in L^2(\R^d)$. Then 
		\[ \norm[\Big]{\sup\limits_{k\in\Z}\; \abs[\Big]{ \sum\limits_{i>k}K^i\ast g}}_2\lesssim \|g\|_2.\]
		\label{Fan}
	\end{lemma}  
	
	We will show that $\norm{\sup\limits_{k\in\Z}\;\abs{\mathcal{H}_{k,j}}}_2^2\lesssim\alpha\norm{f}_1$, for $j=1,2,3$.
	Then by Chebyshev's inequality, we have $\abs{\{x\in E^c:\sup\limits_{k\in\Z}\;\abs{\mathcal{H}_{k,j}(x)}>\alpha/5\}}\lesssim \frac{1}{\alpha}\norm{f}_1$ for $j=1,2,3$.\\
	First, we observe that $\norm{\sup\limits_{k\in\Z}\;\abs{\mathcal{H}_{k,1}}}_2^2\lesssim\norm{g}_2^2\lesssim\alpha\norm{f}_1$, where the first inequality follows from \Cref{Fan}. 
	
	To deal with $\mathcal{H}_{k,j},\;j=2,3$, we need $L^2$ estimates for the following intermediary operators defined as
	\begin{equation}
		\begin{aligned}
			T_{-1}^*h&=\sup\limits_k\;\abs{\sum\limits_{i>k}K_0^i*h},\\
			T_m^*h&=\sup\limits_{k\in\Z}\;\abs[Big]{\sum\limits_{i>k}(K_{m}^i-K_{m-1}^i)*h},\;m\in\N.
		\end{aligned}
	\end{equation}
	The maximal operator $T_{-1}^*$ behaves as maximally truncated Calder\'{o}n-Zygmund operator. Precisely, we have the following.
	\begin{lemma}\label{T_{-1}}
		For $\Omega\in L\log L(\Sp^{d-1})$, the operator $T_{-1}^*$	is bounded in $L^2$ and weak type $(1,1)$.
	\end{lemma}
	\begin{proof}
		We define the corresponding singular integral operator $T_{-1}$ as
		\[T_{-1}f=\sum\limits_{i\in\Z}K_0^i*f=S*f.\]
		We first show that the kernel $S$ satisfies the growth and H\"{o}rmander conditions:
		\[\abs{S(x)}\lesssim\frac{\norm{\Omega}_1}{\abs{x}^d},\;\;\abs{S(x)-S(y)}\lesssim\norm{\Omega}_1\frac{\abs{y}}{\abs{x}^{d+1}}\;\text{for}\;2\abs{y}<\abs{x}.\]
		Indeed we have,
		\begin{align*}
			\abs{K_0^i(x)}&=\abs[\Big]{\int\abs{y}^{-d}\Omega(\frac{y}{\abs{y}})\beta_i(y)\phi_{-i}(x-y)\;dy}\\
			&\lesssim\chi_{2^{i-2}\leq\abs{x}\leq2^{i+2}}(x)\frac{1}{\abs{x}^d}\int\limits_{2^{i-1}\leq\abs{y}\leq2^{i+1}}\abs{\Omega(\frac{y}{\abs{y}})\phi(2^{-i}(x-y))}2^{-id}\;dy\\
			&\lesssim\chi_{2^{i-2}\leq\abs{x}\leq2^{i+2}}(x)\frac{1}{\abs{x}^d}\int_{\Sp^{d-1}}\abs{\Omega(\theta)}\int\limits_{r=0}^{2}\abs{\phi(2^{-i}x-r\theta)}r^{d-1}\;dr\;d\theta\\
			&\lesssim\frac{\norm{\Omega}_1}{\abs{x}^d}\chi_{2^{i-2}\leq\abs{x}\leq2^{i+2}}(x).
		\end{align*}
		For H\"{o}rmander's condition, we observe that
		\begin{align*}
			\abs{\nabla K_0^i(x)}&\leq\int\abs{y}^{-d}\abs[\Big]{\Omega(\frac{y}{\abs{y}})\beta_i(y)\nabla\phi_{-i}(x-y)}\;dy\\
			&\lesssim\chi_{2^{i-2}\leq\abs{x}\leq2^{i+2}}(x)\frac{1}{\abs{x}^d}\int\limits_{2^{i-1}\leq\abs{y}\leq2^{i+1}}\abs{\Omega(\frac{y}{\abs{y}})\nabla\phi(2^{-i}(x-y))}2^{-i(d+1)}\;dy\\
			&\lesssim\frac{\norm{\Omega}_1}{\abs{x}^{d+1}}\chi_{2^{i-2}\leq\abs{x}\leq2^{i+2}}(x).
		\end{align*}
		Thus we get $\abs{\nabla S(x)}\leq\sum\limits_{i\in\Z}\abs{\nabla K_0^i(x)}\lesssim\frac{\norm{\Omega}_1}{\abs{x}^{d+1}}$.
		Therefore by mean value theorem, we have $\abs{S(x)-S(y)}\lesssim\norm{\Omega}_1\frac{\abs{y}}{\abs{x}^{d+1}}\;\text{for}\;2\abs{y}<\abs{x}$.\\
		We also note that,
		\[\abs{\widehat{S}(\xi)}\leq\sum\limits_{i\in\Z}\abs{\widehat{K^i}(\xi)\widehat{\phi_{-i}}(\xi)}\lesssim\sum\limits_{i:\abs{2^i\xi}\leq 1}\abs{2^i\xi}^{\alpha}+\sum\limits_{i:\abs{2^i\xi}>1}\abs{2^i\xi}^{-\alpha}\lesssim 1.\]
		Hence $T_{-1}$ is a $L^2$ bounded Calder\'{o}n-Zygmund operator. And by standard Calder\'{o}n-Zygmund theory \cite{G}, the maximally truncated singular integral operator $T_{-1}^{\sharp}$ is weak (1,1), where $T_{-1}^{\sharp} f=\sup\limits_{\epsilon>0}\;\abs{S_{\epsilon}*f}$ and $S_\epsilon=S\chi_{\abs{.}>\epsilon}$.
		\begin{align}
			\text{Now }T_{-1}^*f(x)&\leq\sup\limits_{k\in\Z}\;\Bigl(\abs[\Big]{\int\limits_{\abs{x-y}>2^{k+1}}\sum\limits_{i>k}K_0^i(x-y)f(y)\;dy}+\abs[\Big]{\int\limits_{\abs{x-y}\leq2^{k+1}}\sum\limits_{i>k}K_0^i(x-y)f(y)\;dy}\Bigr)\nonumber\\
			&\leq\sup\limits_{k\in\Z}\;\Bigl(\abs[\Big]{\int\limits_{\abs{x-y}>2^{k+1}}\sum\limits_{i>k}K_0^i(x-y)f(y)\;dy}+\abs[\Big]{\int\limits_{\abs{x-y}\leq2^{k+1}}K_0^{k+1}(x-y)f(y)\;dy}\Bigr)\nonumber\\
			&\lesssim T_{-1}^{\sharp}f+Mf\label{CZ+M}.
		\end{align}
		Hence $T_{-1}^*$ is also $L^2$-bounded and weak type (1,1).
	\end{proof}
	Now we provide an $L^2$ estimate for the operator $T_m^*$. The following lemma is a $L\log L$ counterpart for Lemma 8 in \cite{H}. The estimate was used in \cite{DHL} to provide a sparse domination for $T_{\Omega}^*$. The proof follows in verbatim. 
	\begin{lemma}\label{l2}
		There exists constant $c_2>0$ such that,
		\[\norm{T_m^*h}_2\lesssim 2^{-c_2 2^m}\norm{h}_2,\;m\in\N.\]
	\end{lemma}
	We now complete the estimates for $\mathcal{H}_{k,j},\;j=2,3$. Using \Cref{l2} and choosing $c_1$ such that $c_1<c_2$, we have
	\begin{align*}
		\norm{\sup\limits_{k\in\Z}\abs{\mathcal{H}_{k,2}}}_2&\lesssim\sum\limits_{n\geq 1}\norm[\Big]{\sup\limits_{k}\;\abs[\Big]{\sum\limits_{i>k}(K^i-K_n^i)*f^n}}_2\\
		&=\sum\limits_{n\geq 1}\norm[\Big]{\sup\limits_{k}\;\abs[\Big]{\sum\limits_{i>k}\sum\limits_{m=n+1}^{\infty}(K_m^i-K_{m-1}^i)*f^n}}_2\\
		&\leq\sum\limits_{n\geq 1}\sum\limits_{m=n+1}^{\infty}\norm{T_m^*(f^n)}_2\\
		&\lesssim\sum\limits_{n\geq 1}\sum\limits_{m=n+1}^{\infty}2^{-c_2 2^m}\norm{f^n}_2\\
		&\lesssim\sum\limits_{n\geq 1}2^{-c_2 2^n}(2^{c_1 2^n}\alpha\norm{f}_1)^{\frac{1}{2}}\\
		&\lesssim (\alpha\norm{f}_1)^{\frac{1}{2}}.
	\end{align*}
	For j=3, we write $K_n^i-K_0^i=\sum\limits_{m=1}^{n}(K_m^i-K_{m-1}^i)$ and use \Cref{l2} to get
	\begin{align*}
		\norm{\sup\limits_{k\in\Z}\abs{\mathcal{H}_{k,3}}}_2&\lesssim\norm[\Big]{\sup\limits_{k}\;\abs[\Big]{\sum\limits_{i>k}\sum\limits_{n=1}^{\infty}\sum\limits_{m=1}^{n}(K_m^i-K_{m-1}^i)*g^n}}_2\\
		&\leq\sum\limits_{m=1}^{\infty}\norm[\Big]{T_m^*(\sum\limits_{n=m}^{\infty}g^n)}_2\\
		&\lesssim\sum\limits_{m=1}^{\infty}2^{-c_2 2^m}\norm[\Big]{\sum\limits_{n=m}^{\infty}g^n}_2\\
		&\lesssim (\alpha\norm{f}_1)^{\frac{1}{2}}.
	\end{align*}
	The estimate for $\abs{\{x\in E^c:\sup\limits_k\;\abs{\mathcal{H}_{k,4}(x)}>\alpha/5\}}$ follows from weak $(1,1)$ boundedness of $T_{-1}^*$ (\Cref{T_{-1}}) and $\norm{\sum\limits_{n\in\N}f^n}_1\lesssim\norm{f}_1$.
	\subsection{Estimates of Seeger.}
	Before estimating the bad part $\mathcal{H}_{k,b}$, we state the estimates in \cite{S1}, that were vital in proving weak type endpoint boundedness of  $T_{\Omega}$.\\
	Let $\{H^i\}_{i\in\Z}$ be a sequence of kernels satisfying,
	\begin{equation}\label{S1}
		{\rm supp}\;(H^i)\subset\{2^{i-2}\leq\abs{x}\leq 2^{i+2}\},
	\end{equation}
	And for each $N\in\N$,
	\begin{equation}\label{S2}
		\sup\limits_{0\leq l\leq N}\sup\limits_{i\in\Z}r^{n+l}\abs[\Big]{\frac{\partial^l}{\partial r^l}H^i(r\theta)}\lesssim C_N,
	\end{equation}
	uniformly in $\theta\in \Sp^{d-1}$ and $r\in\R^+$.\\
	Using the decomposition of $H^i$ in Section 2 of \cite{S1}, we write
	\[H^i=\Gamma_s^i+(H^i-\Gamma_s^i),\]
	and we have the following decay estimates.
	\begin{lemma}[\cite{S1}]\label{Seeger}
		Let $I\subseteq\Z$ be an index set and $\mathcal{Q}$ be a collection of disjoint dyadic cubes. Let $F_s=\sum\limits_{Q\in\mathcal{Q}:l(Q)=2^s}b_{_Q}$, where $b_{_Q}$ is supported in $Q$ and $\norm{b_{_Q}}_1\lesssim \alpha\abs{Q}$. Then there exists $\delta_1>0$ such that for $s>3$, we have
		\begin{equation}\label{Seeger1}
			\norm[\Big]{\sum\limits_{i\in I}\Gamma_s^i*F_{i-s}}_{2}\lesssim C_0 2^{-\delta_1s}\alpha^{\frac{1}{2}}\Bigl(\sum\limits_{Q\in\mathcal{Q}}\norm{b_{_Q}}_{1}\Bigr)^{\frac{1}{2}}.
		\end{equation}
		If we also assume $\int b_{_Q}=0$, then
		\begin{equation}\label{Seeger2}
			\norm{(H^i-\Gamma_s^i)*b_{_Q}}_1\lesssim (C_0+C_{5d})2^{-\delta_1 s}\alpha\norm{b_{_Q}}_1.
		\end{equation}
	\end{lemma}
	For the proof of \Cref{Seeger}, we refer to the proofs of Lemma 2.1 and Lemma 2.2 of \cite{S1}. We note that the aforementioned proofs also work for any index $I\subset\Z$.

	\subsection{Estimate for Bad part}\label{bad}
	We now diverge from the proof of \cite{STW} in the sense that instead of conducting our analysis on the kernel side, we focus on the scales $s$ of the bad part. In fact, we have a much stronger decay in $s$ for the rough singular integrals than the more general Radon transforms  considered in \cite{STW}. 

	We begin by collecting cubes of same scales, namely  \[B_s^n=\sum\limits_{Q\in\mathcal{F}:l(Q)=2^s}b_{_Q}^n.\] Hence  $b^n=\sum\limits_{s\in\Z}B_s^n.$  
	Since ${\rm supp}(K_n^i)\subset \{2^{i-2}\leq\abs{x}\leq 2^{i+2}\},\forall n\geq 0$, we have
	\begin{equation}\label{support}
		\sum\limits_{i>k}\sum\limits_{n\in\N}\sum\limits_{s<3}(K_n^i-K_0^i)*B_{i-s}^n(x)=0,\;\text{for }x\in E^c.
	\end{equation}
	Therefore it remains to estimate $\abs[\Big]{\set[\Big]{x\in\R^d:\sup\limits_k\;\abs[\Big]{\sum\limits_{i>k}\sum\limits_{n\in\N}\sum\limits_{s\geq 3}(K_n^i-K_0^i)*B_{i-s}^n(x)}>\frac{\alpha}{5}}}$.\\
	
	Let $C>0$ be a constant to be chosen later.  By Chebyshev's inequality and $\norm{K_n^i}_1\lesssim\norm{\Omega}_{L^1(\Sp^{d-1})}$, we have
	\begin{align*}
		&\abs[\Big]{\set[\Big]{x\in\R^d:\sup\limits_k\;\abs[\Big]{\sum\limits_{i>k}\sum\limits_{n\in\N}\sum\limits_{s=3}^{Cn}(K_n^i-K_0^i)*B_{i-s}^n(x)}>\frac{\alpha}{10}}}\\
		&\lesssim\frac{1}{\alpha}\sum\limits_{i\in\Z}\sum\limits_{n\in\N}\sum\limits_{s=3}^{Cn}(\norm{K_n^i}_1+\norm{K_0^i}_1)\norm{B_{i-s}^n}_1\\
		&\lesssim\frac{1}{\alpha}\sum\limits_{n\in\N}\sum\limits_{s=3}^{Cn}\sum\limits_{i\in\Z}\sum\limits_{Q:l(Q)=2^{i-s}}\norm{b_{_Q}^n}_1\\
		&\lesssim\frac{1}{\alpha}\sum\limits_{n\in\N}n\sum\limits_{Q\in\mathcal{F}}\norm{f_{_Q}^n}_1\\
		&\lesssim\frac{1}{\alpha}\sum\limits_{n\in\N}\sum\limits_{Q\in\mathcal{F}}\int\abs{f_{_Q}^n}\log\log\Biggl(e^2+\frac{\abs{f_{_Q}^n}}{\alpha}\Biggr)\\
		&\lesssim\frac{1}{\alpha}\sum\limits_{Q\in\mathcal{F}}\int\abs{f_{_Q}}\log\log\Biggl(e^2+\frac{\abs{f_{_Q}}}{\alpha}\Biggr)\\
		&\lesssim\frac{1}{\alpha}\int \abs{f(x)}\log\log\left(e^2+\frac{\abs{f(x)}}{\alpha}\right)\;dx.
	\end{align*}
	To deal with $\Omega\in L\log L(\Sp^{d-1})$, we decompose the kernel into bounded and integrable parts. In this regard, we define the set $\mathcal{D}_s=\{\theta\in \Sp^{d-1}:\abs{\Omega(\theta)}\leq2^{\delta s}\norm{\Omega}_{L^1(\Sp^{d-1})}\}$, where the constant $\delta>0$ is to be chosen later. We write \[\Omega=\Omega\chi_{\mathcal{D}_s^c}+\Omega\chi_{\mathcal{D}_s}=\Omega^++\Omega^-\]
	And thus $K_n^i=K_n^{i+}+K_n^{i-}$, where $K_n^{i\pm}=K^{i\pm}*\phi_{2^n-i}$ and $K^{i\pm}(x)=\abs{x}^{-d}\Omega^{\pm}(\frac{x}{\abs{x}})\beta_i(x)$.
	Therefore we have,
	\begin{align*}
		&\abs[\Big]{\set[\Big]{x\in\R^d:\sup\limits_k\;\abs[\Big]{\sum\limits_{i>k}\sum\limits_{n\in\N}\sum\limits_{s=Cn}^{\infty}(K_n^{i}-K_0^{i})*B_{i-s}^n(x)}>\frac{\alpha}{10}}}\\
		\leq&\abs[\Big]{\set[\Big]{x\in\R^d:\sup\limits_k\;\abs[\Big]{\sum\limits_{i>k}\sum\limits_{n\in\N}\sum\limits_{s=Cn}^{\infty}(K_n^{i+}-K_0^{i+})*B_{i-s}^n(x)}>\frac{\alpha}{20}}}\\
		&+\abs[\Big]{\set[\Big]{x\in\R^d:\sup\limits_k\;\abs[\Big]{\sum\limits_{i>k}\sum\limits_{n\in\N}\sum\limits_{s=Cn}^{\infty}(K_n^{i-}-K_0^{i-})*B_{i-s}^n(x)}>\frac{\alpha}{20}}}\\
		=&\mathcal{B}^++\mathcal{B}^-
	\end{align*}
	To bound $\mathcal{B}^+$, by Chebyshev's inequlity we have
	\begin{eqnarray*}
		\mathcal{B}^+&\lesssim&\frac{1}{\alpha}\sum\limits_{i\in\Z}\sum\limits_{n\in\N}\sum\limits_{s=Cn}^{\infty}\norm{(\abs{K_n^{i+}}+\abs{K_0^{i+}})*\abs{B_{i-s}^n}}_1\\
		&\lesssim&\frac{1}{\alpha}\sum\limits_{n\in\N}\sum\limits_{s=3}^{\infty}\sum\limits_{i\in\Z}\norm{B_{i-s}^n}_1\norm{K^{i+}}_1\\
		&\lesssim&\frac{1}{\alpha}\sum\limits_{n\in\N}\sum\limits_{Q\in\mathcal{F}}\norm{b_{_Q}^n}_1\sum\limits_{s=3}^{\infty}\int\limits_{\{\abs{\Omega(\theta)}>2^{\delta s}\norm{\Omega}_{L^1(\Sp^{d-1})}\}}\abs{\Omega(\theta)}\;d\theta\\
		&\lesssim&\frac{1}{\alpha}\norm{f}_1\int\limits_{\Sp^{d-1}}\abs{\Omega(\theta)}\;\text{card}\set[\Big]{s\in\N:2^{\delta s}<\frac{\abs{\Omega(\theta)}}{\norm{\Omega}_1}}\;d\theta\\
		&\lesssim&\frac{1}{\alpha}\norm{f}_1\norm{\Omega}_{L\log L(\Sp^{d-1})}.
	\end{eqnarray*}
	To estimate $\mathcal{B}^-$, we expand the kernel telescopically as
	\begin{equation}
		K_n^{i-}-K_0^{i-}=\sum\limits_{m=1}^n(K_m^{i-}-K_{m-1}^{i-})=\sum\limits_{m=1}^nH_m^i.
		\label{hmi}
	\end{equation}
	By a change of variable and differentiating in radial variable one can verify that the kernel $H_m^i$ satisfies the estimates (\ref{S1}) and (\ref{S2}) with $C_N=2^{\delta s}$ (or see \cite{H}).
	Now, to deal with the maximal operator in $\mathcal{B}^-$ we need a Cotlar type inequality that was first proved by Honz\'{i}k \cite{H}. But in contrast with \cite{H}, we group the indices $i$ with respect to the size of the function. We provide a proof for the convenience of the reader. We need the following  $L^1$ estimate which follows  from the support condition of $\phi$.
	\begin{lemma}\label{phi}
		Let $a_1,a_2,a_3\in\Z$ with $a_1<a_2<a_3$. Then there exists $\delta_2>0$ such that
		\begin{align*}
			\int\abs{\phi_{a_2}*\phi_{a_1}-\phi_{a_1}}\lesssim 2^{\delta_2(a_1-a_2)},\hspace{1cm}&{\rm supp}(\phi_{a_2}*\phi_{a_1}-\phi_{a_1})\subseteq B(0,2^{-da_1})\text{   and}\\
			\int\abs{\phi_{a_1}*(\phi_{a_3}-\phi_{a_2})}\lesssim 2^{\delta_2(a_1-a_2)},\hspace{1cm}&{\rm supp}(\phi_{a_1}*(\phi_{a_3}-\phi_{a_2}))\subseteq B(0,2^{-da_1}).
		\end{align*}
	\end{lemma}
	\begin{lemma}\label{Cotlar}
		There exists a sequence of non-negative functions $\{v_i\}_{i\in\Z}$ with $\sup\limits_{i\in\Z}\;\norm{v_i}_{1}\lesssim 1$ such that the following pointwise inequality holds:
		\begin{equation}\label{CI}
			\begin{aligned}
				\sup\limits_{k\in\Z}\;\abs[\Big]{\sum\limits_{i>k}H_{m}^i*B_{i-s}^n}\lesssim\sum\limits_{r=0}^{s2^{n+2}-1}&M\Bigl(\sum\limits_{i\equiv r(mod\; s2^{n+2})}H_{m}^i*B_{i-s}^n\Bigr)+ \\ &C_{n,s}\sum\limits_{i\in\Z}\sum\limits_{Q\in\mathcal{F}:l(Q)=2^{i-s}}v_i*\abs{b_{_Q}^n},
			\end{aligned}
		\end{equation}
		where $M$ is the Hardy-Littlewood maximal function, $C_{n,s}= s2^n2^{-(\delta_2-\delta)s}$ and $\delta_2$ is as in \Cref{phi}. 
	\end{lemma}
	\begin{proof}
		By Euclid's algorithm, we write $i=ps2^{n+2}+r,\;k=qs2^{n+2}+r'$ with $0\leq r,r'\leq s2^{n+2}-1$ and $i>k$ implies $p\geq q$. Hence we have
		\begin{align*}
			&\sup\limits_{k\in\Z}\;\abs[\Big]{\sum\limits_{i>k}H_{m}^i*B_{i-s}^n}\\
			&\leq\sup\limits_{q\in\Z}\sup\limits_{r'=0,1,\dots,s2^{n+2}-1}\;\abs[\Big]{\sum\limits_{\substack{p,r:\\ps2^{n+2}+r>qs2^{n+2}+r'}}\sum\limits_{\substack{Q\in\mathcal{F}:\\l(Q)=2^{ps2^{n+2}+r-s}}}H_m^{ps2^{n+2}+r}*b_{_Q}^n}\\
			&\leq\sum\limits_{r=0}^{s2^{n+2}-1}\sup\limits_{q\in\Z}\;\abs[\Big]{\sum\limits_{p\geq q}\sum\limits_{\substack{Q\in\mathcal{F}:\\l(Q)=2^{ps2^{n+2}+r-s}}}H_m^{ps2^{n+2}+r}*b_{_Q}^n}.
		\end{align*}
		Now we fix a $r$ and write
		\begin{align*}
			&\abs[\Big]{\sum\limits_{p\geq q}\sum\limits_{\substack{Q\in\mathcal{F}:\\l(Q)=2^{ps2^{n+2}+r-s}}}H_m^{ps2^{n+2}+r}*b_{_Q}^n}\\\leq&\abs[\Big]{\sum\limits_{p\geq q}\sum\limits_{\substack{Q\in\mathcal{F}:\\l(Q)=2^{ps2^{n+2}+r-s}}}(H_m^{ps2^{n+2}+r}-\phi_{t_q}*H_m^{ps2^{n+2}+r})*b_{_Q}^n}\\
			&+\abs[\Big]{\sum\limits_{p\in\Z}\sum\limits_{\substack{Q\in\mathcal{F}:\\l(Q)=2^{ps2^{n+2}+r-s}}}H_m^{ps2^{n+2}+r}*\phi_{t_q}*b_{_Q}^n}\\
			&+\abs[\Big]{\sum\limits_{p<q}\sum\limits_{\substack{Q\in\mathcal{F}:\\l(Q)=2^{ps2^{n+2}+r-s}}}H_m^{ps2^{n+2}+r}*\phi_{t_q}*b_{_Q}^n}\\
			:=&\mathcal{I}_{1,q}+\mathcal{I}_{2,q}+\mathcal{I}_{3,q},
		\end{align*}
		where $t_q=-qs2^{n+2}-r+s2^{n+1}$.
		\vskip 0.1cm
		Clearly $\mathcal{I}_{2,q}\lesssim M(\sum\limits_{p\in\Z}\sum\limits_{\substack{Q\in\mathcal{F}:\\l(Q)=2^{ps2^{n+2}+r-s}}}H_m^{ps2^{n+2}+r}*b_{_Q}^n)$.
		\vskip 0.1cm
		To estimate $\mathcal{I}_{1,q}$, we expand the kernel to have:
		\begin{align*}
			\abs{H_m^{ps2^{n+2}+r}&-\phi_{t_q}*H_m^{ps2^{n+2}+r}}\\
			&\leq\abs{(\phi_{2^{m}-ps2^{n+2}-r}-\phi_{t_q}*\phi_{2^{m}-ps2^{n+2}-r})*K^{(ps2^{n+2}+r)-}}\\
			&+\abs{(\phi_{2^{m-1}-ps2^{n+2}-r}-\phi_{t_q}*\phi_{2^{m-1}-ps2^{n+2}-r})*K^{(ps2^{n+2}+r)-}}.
		\end{align*}
		We estimate the first term, the second follows similarly. For $p\geq q$, we have $t_q> 2^{m}-ps2^{n+2}-r$ and \Cref{phi} implies
		\begin{align*}
			&\abs{(\phi_{2^{m}-ps2^{n+2}-r}-\phi_{t_q}*\phi_{2^{m}-ps2^{n+2}-r})*K^{(ps2^{n+2}+r)-}}\\
			&\lesssim 2^{\delta s}v_{ps2^{n+2}+r}\int\abs{\phi_{2^{m}-ps2^{n+2}-r}-\phi_{t_q}*\phi_{2^{m}-ps2^{n+2}-r}}\\
			&\lesssim 2^{\delta s} 2^{-\delta_2 s} 2^{-\delta_2(p-q)s2^{n+2}}v_{ps2^{n+2}+r},
		\end{align*}
		where $v_i:=2^{-in}\chi_{2^{i-2}\leq\abs{x}\leq 2^{i+2}}$. Therefore we have
		\begin{align*}
			\mathcal{I}_{1,q}&\lesssim 2^{\delta s} 2^{-\delta_2s}\sum\limits_{p\geq q}2^{-\delta_2(p-q)s2^{n+2}}\sum\limits_{\substack{Q\in\mathcal{F}\\l(Q)=2^{ps2^{n+2}+r-s}}}v_{ps2^{n+2}+r}*\abs{b_{_Q}^n}\\
			&\leq 2^{-(\delta_2-\delta)s}\sum\limits_{i\in\Z}\sum\limits_{\substack{Q\in\mathcal{F}\\l(Q)=2^{i-s}}}v_i*\abs{b_{_Q}^n}.
		\end{align*}
		Now we estimate $\mathcal{I}_{3,q}$. $p<q,\;t_q<2^{m-1}-ps2^{n+2}-r$ and \Cref{phi} implies
		\begin{align*}
			\abs{H_m^{ps2^{n+2}+r}*\phi_{t_q}}
			&\leq\abs{(\phi_{2^{m}-ps2^{n+2}-r}- \phi_{2^{m-1}-ps2^{n+2}-r})*\phi_{t_q}*K^{(ps2^{n+2}+r)-}}\\
			&\lesssim 2^{\delta s} v_{ps2^{n+2}+r}\int\abs{(\phi_{2^{m}-ps2^{n+2}-r}- \phi_{2^{m-1}-ps2^{n+2}-r})*\phi_{t_q}}\\
			&\lesssim 2^{\delta s}2^{-\delta_2s}2^{-\delta_2(q-p-1)s2^{n+2}}v_{ps2^{n+2}+r}.
		\end{align*}
		Hence $\mathcal{I}_{3,q}\leq  2^{-(\delta_2-\delta)s} \sum\limits_{i\in\Z}\sum\limits_{\substack{Q\in\mathcal{F}\\l(Q)=2^{i-s}}}v_i*\abs{b_{_Q}^n}.$
		\vskip 0.1cm
		Finally, taking supremum over $q\in\Z$ and summing in $r$ produces the desired inequality.\\
	\end{proof}
	Now we  conclude the proof of \Cref{Main}. Choose $\delta<\frac{1}{2}\min\{\delta_1,\delta_2\}$ and $C\delta>100$. By \Cref{Cotlar} we have,
	\begin{align*}
		\mathcal{B}^-\lesssim&\abs[\Big]{\set[\Big]{x\in \R^d:\sum\limits_{m=1}^{\infty}\sum\limits_{n=m}^{\infty}\sum\limits_{s=Cn}^{\infty}\sup\limits_k\;\abs[\Big]{\sum\limits_{i>k}H_m^i*B_{i-s}^n(x)}>\frac{\alpha}{20}}}\\
		\lesssim&\abs[\Big]{\set[\Big]{x\in \R^d:\sum\limits_{m=1}^{\infty}\sum\limits_{n=m}^{\infty}\sum\limits_{s=Cn}^{\infty}C_{n,s}\sum\limits_{i\in\Z}\sum\limits_{Q\in\mathcal{F}:l(Q)=2^{i-s}}v_i*\abs{b_{_Q}^n(x)}>\frac{\alpha}{40}}}\\
		&+\abs[\Big]{\set[\Big]{x\in \R^d:\sum\limits_{m=1}^{\infty}\sum\limits_{n=m}^{\infty}\sum\limits_{s=Cn}^{\infty}\sum\limits_{r=0}^{s2^{n+2}-1}M\Bigl(\sum\limits_{i\equiv r(mod\;s2^{n+2})}H_m^i*B_{i-s}^n\Bigr)(x)>\frac{\alpha}{40}}}\\
		=&\mathcal{B}_1^- +\mathcal{B}_2^-.
	\end{align*}
	By Chebyshev's inequality we have,
	\begin{align*}
		\mathcal{B}_1^-&\lesssim\frac{1}{\alpha}\sum\limits_{m=1}^{\infty}\sum\limits_{n=m}^{\infty}\sum\limits_{s=Cn}^{\infty}C_{n,s}\sum\limits_{i\in\Z}\sum\limits_{Q\in\mathcal{F}:l(Q)=2^{i-s}}\norm{v_i}_1\norm{b_{_Q}^n}_1\\
		&\lesssim\frac{1}{\alpha}\sum\limits_{m=1}^{\infty}\sum\limits_{n=m}^{\infty}\sum\limits_{s=Cn}^{\infty}C_{n,s}\sum\limits_{i\in\Z}\sum\limits_{Q\in\mathcal{F}:l(Q)=2^{i-s}}\norm{b_{_Q}^n}_1\\
		&\lesssim\frac{1}{\alpha}\sum\limits_{m=1}^{\infty}\sum\limits_{n=m}^{\infty}\sum\limits_{s=Cn}^{\infty}s2^n2^{-\delta s}\sum\limits_{Q\in\mathcal{F}}\norm{b_{_Q}^n}_1\\
		&\lesssim\frac{1}{\alpha}\sum\limits_{m=1}^{\infty}2^{-98m}\sum\limits_{n=m}^{\infty}\sum\limits_{Q\in\mathcal{F}}\norm{b_{_Q}^n}_1\\
		&\lesssim\frac{1}{\alpha}\norm{f}_1.
	\end{align*}
	To deal with $\mathcal{B}_2^-$, we simply break the kernel into $L^1$ and $L^2$ parts. Indeed we have,
	\begin{align*}
		\mathcal{B}_2^-\leq&\abs[\Big]{\set[\Big]{x\in\R^d:\sum\limits_{m=1}^{\infty}\sum\limits_{n=m}^{\infty}\sum\limits_{s=Cn}^{\infty}\sum\limits_{r=0}^{s2^{n+2}-1}M\Bigl(\sum\limits_{i\equiv r(mod\;s2^{n+2})}(\Gamma_{m,s}^i-H_m^i)*B_{i-s}^n\Bigr)(x)>\frac{\alpha}{80}}}\\
		&+\abs[\Big]{\set[\Big]{x\in\R^d:\sum\limits_{m=1}^{\infty}\sum\limits_{n=m}^{\infty}\sum\limits_{s=Cn}^{\infty}\sum\limits_{r=0}^{s2^{n+2}-1}M\Bigl(\sum\limits_{i\equiv r(mod\;s2^{n+2})}\Gamma_{m,s}^i*B_{i-s}^n\Bigr)(x)>\frac{\alpha}{80}}}\\
		=&\mathcal{B}_{2,1}^-+\mathcal{B}_{2,2}^-
	\end{align*}
	The estimate for $\mathcal{B}_{2,1}^-$ follows by positivity and weak type (1,1) boundedness of $M$ and estimate (\ref{Seeger2}). 
	\begin{align*}			
		\mathcal{B}_{2,1}^-&\lesssim\frac{1}{\alpha}\sum\limits_{m=1}^{\infty}\sum\limits_{n=m}^{\infty}\sum\limits_{s=Cn}^{\infty}\sum\limits_{r=0}^{s2^{n+2}-1}\norm[\Big]{\sum\limits_{i\equiv r(mod\;s2^{n+2})}(\Gamma_{m,s}^i-H_m^i)*B_{i-s}^n}_1\\
		&\lesssim\frac{1}{\alpha}\sum\limits_{m=1}^{\infty}\sum\limits_{n=m}^{\infty}\sum\limits_{s=Cn}^{\infty}\sum\limits_{r=0}^{s2^{n+2}-1}\sum\limits_{Q\in\mathcal{F}}2^{-(\delta_1-\delta)s}\norm{b_{_Q}^n}_1\\
		&\lesssim\frac{1}{\alpha}\sum\limits_{m=1}^{\infty}2^{-98m}\sum\limits_{n\in\N}\sum\limits_{Q\in\mathcal{F}}\norm{b_{_Q}^n}_1\\
		&\lesssim\frac{1}{\alpha}\norm{f}_1.
	\end{align*}
	To estimate $\mathcal{B}_{2,2}^-$, we use $L^2$ boundedness of $M$ and estimate (\ref{Seeger1}) to get,
	\begin{align*}
		\mathcal{B}_{2,2}^-&\lesssim\frac{1}{\alpha^2}\Bigl(\sum\limits_{m=1}^{\infty}\sum\limits_{n=m}^{\infty}\sum\limits_{s=Cn}^{\infty}\sum\limits_{r=0}^{s2^{n+2}-1}\norm[\Big]{M\Bigl(\sum\limits_{i\equiv r(mod\;s2^{n+2})}\Gamma_{m,s}^i*B_{i-s}^n\Bigr)}_2\Bigr)^2\\
		&\lesssim\frac{1}{\alpha^2}\Bigl(\sum\limits_{m=1}^{\infty}\sum\limits_{n=m}^{\infty}\sum\limits_{s=Cn}^{\infty}\sum\limits_{r=0}^{s2^{n+2}-1}\norm[\Big]{\sum\limits_{i\equiv r(mod\;s2^{n+2})}\Gamma_{m,s}^i*B_{i-s}^n}_2\Bigr)^2\\
		&\lesssim\frac{1}{\alpha}\Bigl(\sum\limits_{m=1}^{\infty}\sum\limits_{n=m}^{\infty}\sum\limits_{s=Cn}^{\infty}s2^n2^{-\frac{\delta s}{2}}\Bigl(\sum\limits_{Q\in\mathcal{F}}\norm{b_{_Q}^n}_{1}\Bigr)^{\frac{1}{2}}\Bigr)^2\\
		&\lesssim\frac{1}{\alpha}\Bigl(\sum\limits_{m=1}^{\infty}\sum\limits_{n=m}^{\infty}2^{-48n}\Bigl(\sum\limits_{Q\in\mathcal{F}}\norm{b_{_Q}^n}_{1}\Bigr)^{\frac{1}{2}}\Bigr)^2\\
		&\lesssim\frac{1}{\alpha}\Bigl(\sum\limits_{m=1}^{\infty}\Bigl(\sum\limits_{n=m}^{\infty}2^{-96n}\Bigr)^{\frac{1}{2}}\Bigr)^2\Bigl(\sum\limits_{n=1}^{\infty}\sum\limits_{Q\in\mathcal{F}}\norm{b_{_Q}^n}_1\Bigr)\\
		&\lesssim\frac{1}{\alpha}\norm{f}_1,
	\end{align*}
	where we have used Cauchy-Schwartz inequality in the second to last step.
	
	\section{Proof of \Cref{homega}}

	The proof of \Cref{homega} follows the same line of arguments as of \Cref{Main}. In this section we will provide the necessary modifications required.  Here  we denote $K(x)= h(|x|)\Omega\Bigl(\frac{x}{|x|}\Bigr) |x|^{-d}$.   Eq. \eqref{discr} will be replaced by
	\[T^*_{\Omega,h}f\leq \|h\|_\infty M_{\Omega}f+\sup\limits_{k\in\Z}\;\abs{\sum\limits_{i>k}K^i*f}.\]

	We apply the same decomposition of the function $f$ as in \Cref{function} .  In \cite{FP}, it was shown that $\widehat{K^i}$ satisfies the inequality \eqref{Fourierestimate}. Hence, estimation for $\mathcal{H}_{k,1}$  will follow immediately.   Routinely one can check that \Cref{T_{-1}}  and \Cref{l2}  are true in this setting  ( bounds will have a factor of $\|h\|_\infty$ ).  So we have the appropriate estimates for  $\mathcal{H}_{k,j},\;j=2,3.$
	
	For the bad part  we will follow exactly the same argument given   in  \Cref{bad} except for the term $\mathcal{B}_2^-$.  Due to unavailability of Seeger type estimate \eqref{S2} we will follow the approach of  Christ and Rubio de Francia \cite{CR}.   
	
	\begin{equation*}
		H_m^i= K_m^{i-}-K_{m-1}^{i-}
	\end{equation*}
	It is enough to handle $\norm{\sum\limits_{i\in I}K_m^{i-}*B^n_{i-s}}_2$. Now,
	\begin{eqnarray*}
		\norm{\sum\limits_{i\in I}K_m^{i-}*B^n_{i-s}}_2^2 &=&  \sum\limits_{i\in I}\norm{K_m^{i-}*B^n_{i-s}}_2^2 + 2\sum\limits_{\substack{i,i'\in I\\ i\leq i'-3}} \langle K_m^{i-}\ast B^n_{i-s},  K_m^{i'-}\ast B^n_{i'-s}\rangle\\
		&& + 2\sum\limits_{\substack{i,i'\in I\\  i'-3< i<i'}} \langle K_m^{i-}\ast B^n_{i-s},  K_m^{i'-}\ast B^n_{i'-s}\rangle\\
	\end{eqnarray*}
	By Cauchy Schwartz inequality the third term is dominated by the first. Recall $\|\phi_i\|_1=1$.  Therefore,
	\begin{equation}
		\norm{K_m^{i-}*B^n_{i-s}}_2^2\leq  \norm{K^{i-}*B^n_{i-s}}_2^2 \lesssim 2^{-2(\delta_1-\delta) s}\alpha \|h\|_\infty^2\| B^n_{i-s}\|_1,
		\label{CR1}
	\end{equation}
	where the last inequality follows from  the argument of inequality 2.3 of \cite{CR}. For the cross terms 
	\begin{eqnarray*}
		\Bigl| \langle K_m^{i-}\ast B^n_{i-s},  K_m^{i'-}\ast B^n_{i'-s}\rangle\Bigr| &= &\Bigl|\langle K^{i-}\ast \tilde{K}^{i'-}\ast B^n_{i-s},  \tilde{\phi}_{2^m-i}\ast \phi_{2^m-i'}\ast B^n_{i'-s}\rangle\Bigr| \\
		&\leq& \|B^n_{i'-s}\|_1 \|K^{i-}\ast \tilde{K}^{i'-}\ast B^n_{i-s}\|_\infty\\
	\end{eqnarray*}
	By using Lemma 6.1 of \cite{CR}, summing in $i,i'$  and \Cref{CR1} we get the following. 
	\begin{lemma}
		\[\norm{\sum\limits_{i\in I}H_m^i*B^n_{i-s}}_2\lesssim 2^{-(\delta_1-\delta) s}\alpha^{\frac{1}{2}}\|h\|_\infty\Bigl(\sum\limits_{Q\in\mathcal{Q}}\norm{b_{_Q}^n}_{1}\Bigr)^{\frac{1}{2}}\]
	\end{lemma}
	In view of the above lemma we conclude the proof of \Cref{homega},
	\begin{align*}
		\mathcal{B}_{2}^-
		&\lesssim\frac{1}{\alpha^2}\Bigl(\sum\limits_{m=1}^{\infty}\sum\limits_{n=m}^{\infty}\sum\limits_{s=Cn}^{\infty}\sum\limits_{r=0}^{s2^{n+2}-1}\norm[\Big]{\sum\limits_{i\equiv r(mod\;s2^{n+2})} H_{m}^i*B_{i-s}^n}_2\Bigr)^2\\
		&\lesssim\frac{1}{\alpha}\Bigl(\sum\limits_{m=1}^{\infty}\sum\limits_{n=m}^{\infty}\sum\limits_{s=Cn}^{\infty}s2^n2^{-\frac{\delta s}{2}}\Bigl(\sum\limits_{Q\in\mathcal{F}}\norm{b_{_Q}^n}_{1}\Bigr)^{\frac{1}{2}}\Bigr)^2\\
		&\lesssim\frac{1}{\alpha}\Bigl(\sum\limits_{m=1}^{\infty}\Bigl(\sum\limits_{n=m}^{\infty}2^{-96n}\Bigr)^{\frac{1}{2}}\Bigr)^2\Bigl(\sum\limits_{n=1}^{\infty}\sum\limits_{Q\in\mathcal{F}}\norm{b_{_Q}^n}_1\Bigr)\\
		&\lesssim\frac{1}{\alpha}\norm{f}_1.
	\end{align*}
	\section{Proof of \Cref{Mainw}}
	Let $f\in L\log\log L(\R^d)$ and  $w\in A_1$.   It is enough to show the appropriate boundedness for $\sup\limits_{k\in\Z}\;\abs{\sum\limits_{i>k}K^i*f}$, taking into account \Cref{discr},  $M_{\Omega}f(x)\leq\norm{\Omega}_{L^{\infty}}Mf(x)$  a.e. and $\norm{M}_{L^1(w)\to L^{1,\infty}(w)}\lesssim [w]_{A_1}$ (see \cite{G}).
	
	We follow the exact decomposition of the function $f$  as in \Cref{function}.  In this weighted setting, we need a modification of the decomposition of the kernel. Let $\epsilon_w>0$ be defined as
	\[\epsilon_w=\frac{\log([w]_{A_\infty}+1)}{2(1+c_d[w]_{A_\infty}(1+\log([w]_{A_\infty}+1)))}\]
	and $\Delta$ be the greatest integer less than or equal to ${\epsilon^{-1}_w}$. For this section, we denote $K_n^i$ as 
	\[K_n^i=K^i*\phi_{_{\Delta2^n-i}},\;n\in\N,\]
	and $K_0^i$ is defined as earlier.
	We decompose our operator as in \eqref{break} with the modified $K_n^i$. The exceptional set $E=\cup_{_{Q\in\mathcal{F}}}(100d)Q$ is estimated as follows
	\[w(E)\lesssim\sum\limits_{Q\in\mathcal{F}}\frac{w((100d)Q)}{\abs{(100d)Q}}\abs{Q}\leq\frac{[w]_{A_1}}{\alpha}\sum\limits_{Q\in\mathcal{F}}\inf\limits_{y\in(100d)Q} w(y) \int_Q f \leq\frac{[w]_{A_1}}{\alpha}\norm{f}_{L^1(w)}.\]
	To estimate $\mathcal{H}_{k,j},j=1,2,3$, the good parts, we need appropriate Fefferman-Stein type inequalities. The following inequality was proved in \cite{DHL} (Theorem 1.3).
	\begin{lemma}[\cite{DHL}]\label{FSi1}
		Let $1<r<p$, then the following is true:
		\begin{equation}
			\norm[\Big]{\sup\limits_k\;\abs[\Big]{\sum\limits_{i>k}K^i*g}}_{L^p(w)}\lesssim p^2(p')^{\frac{1}{p}}(r')^{1+\frac{1}{p'}}\norm{g}_{L^p(M_rw)},
		\end{equation}
		where $M_rw=M(w^r)^{\frac{1}{r}}$.
	\end{lemma}
	We need  a Fefferman-Stein inequality for $T_m^*$. By applying the abstract result, Theorem 3.3 of \cite{DHL} (taking into consideration of their calculation of sparse constants in page 1890) to $T_m^*$ we  get $(1+\epsilon_w,1+\epsilon_w)$ sparse domination, namely we have
	\[\abs{\langle T_m^*f,g\rangle}\lesssim \frac{2^{-c_2 2^{m-1}}}{\epsilon_w}\sum\limits_{\mathcal{Q}\in\mathcal{S}}\abs{\mathcal{Q}}^{-\frac{1-\epsilon_w}{1+\epsilon_w}}\norm{f}_{L^{1+\epsilon_w}(\mathcal{Q})}\norm{g}_{L^{1+\epsilon_w}(\mathcal{Q})},\]
	where $\mathcal{S}$ is a $\frac{1}{2}$-sparse family of cubes i.e. for each $\mathcal{Q}\in\mathcal{S}$ there exists $E_{\mathcal{Q}}\subset\mathcal{Q}$ such that $\abs{E_{\mathcal{Q}}}\geq \frac{1}{2}\abs{\mathcal{Q}}$ and $\{E_{\mathcal{Q}}:\mathcal{Q}\in\mathcal{S}\}$ are pairwise disjoint. Now, following exactly the proof of Theorem 1.3 of \cite{DHL} and keeping track of the constant depending on $m$, we obtain the required inequality.
	\begin{lemma}\label{FSi2}
		Let $1<r<p$ be such that $(r-1)=2\epsilon_w(pr-r+1)$. Then we have
		\begin{equation}
			\norm{T_m^*g}_{L^p(w)}\lesssim 2^{-c_2 2^{m-1}}p^2(p')^{\frac{1}{p}}(r')^{1+\frac{1}{p'}}\norm{g}_{L^p(M_rw)},
		\end{equation}
		where $c_2>0$ as in \Cref{l2}.
	\end{lemma}
	We also state the optimal reverse H\"{o}lder inequality obtained in \cite{HPR}.
	\begin{lemma}[\cite{HPR}]\label{Rhi}
		Let $w\in A_{\infty}$. Then there exists an absolute constant $c_d>0$ such that for any $r\in[1,1+\frac{1}{c_d[w]_{A_\infty}}]$ and cube $Q$, we have
		\[\Bigl(\frac{1}{\abs{Q}}\int_Q w^r\Bigr)^{\frac{1}{r}}\leq \frac{2}{\abs{Q}}\int_Q w.\]
		Consequently, the following pointwise inequality is true:
		\[M_{r_w}w\lesssim [w]_{A_1}w,\] 
		where $r_w=1+\frac{1}{c_d[w]_{A_\infty}}$.
	\end{lemma}
	To estimate the  bad part, for $\lambda>0, s\geq 3$,  we require some estimates of the measure of set $E_\lambda^s$  which is defined as
	\[E_\lambda^s=\set[\Big]{x\in \R^d: \sum\limits_{n=1}^{s/Cs_w}\sup\limits_k\;\abs[\Big]{\sum\limits_{i>k}(K_n^i-K_0^i)*B_{i-s}^n(x)}>\lambda\alpha},\]
	where $C>0$ to be chosen later and $s_w=[w]_{A_\infty}\log([w]_{A_\infty}+1)$. The first is a unweighted measure estimate with a decay in $s$ and the other is a simple weighted estimate.
	\begin{lemma}\label{w_est1}
		Let $0<\lambda<1$ and $0<\delta<\min(\delta_1,\delta_2)/3$, where $\delta_1,\delta_2$ as in \Cref{Seeger} and \Cref{phi} respectively. Then for any non-negative weight $v$, we have
		\begin{align*}
			\abs{E_\lambda^s}&\lesssim \frac{1}{\lambda^2}2^{-\delta s}\sum\limits_{Q\in\mathcal{F}}\abs{Q},\\
			v(E_\lambda^s)&\lesssim\frac{1}{\lambda}\sum\limits_{Q\in\mathcal{F}}\abs{Q}\inf\limits_Q Mv.
		\end{align*}
	\end{lemma}
	\begin{proof}
		We first observe that for the modified kernel $K_n^i-K_0^i=\sum\limits_{m=1}^n(K_m^i-K_{m-1}^i)=\sum\limits_{m=1}^nH_m^i$, the corresponding Cotlar-type inequality \eqref{CI} is given by
		\begin{align*}
			\sup\limits_{k\in\Z}\;\abs[\Big]{\sum\limits_{i>k}H_{m}^i*B_{i-s}^n}\lesssim\sum\limits_{r=0}^{\Delta s2^{n+2}-1}&M\Bigl(\sum\limits_{i\equiv r(mod\; \Delta s2^{n+2})}H_{m}^i*B_{i-s}^n\Bigr) \\ +&\Delta s2^n2^{-\delta_2s}\sum\limits_{i\in\Z}\sum\limits_{Q\in\mathcal{F}:l(Q)=2^{i-s}}v_i*\abs{b_{_Q}^n},
		\end{align*}
		Therefore in regard to the above inequality and the decomposition in Section \S 2.3, we get
		\begin{align*}
			\abs{E_\lambda^s}\lesssim&\abs[\Big]{\set[\Big]{x\in \R^d: \sum\limits_{n=1}^{s/Cs_w}\sum\limits_{m=1}^n\Delta s2^n2^{-\delta_2s}\sum\limits_{i\in\Z}\sum\limits_{Q\in\mathcal{F}:l(Q)=2^{i-s}}v_i*\abs{b_{_Q}^n}(x)>\frac{\lambda\alpha}{3}}}\\
			&+\abs[\Big]{\set[\Big]{x\in \R^d:\sum\limits_{n=1}^{s/Cs_w}\sum\limits_{m=1}^n\sum\limits_{r=0}^{\Delta s2^{n+2}-1}M\Bigl(\sum\limits_{i\equiv r(mod\;\Delta s2^{n+2})}(\Gamma_{m,s}^i-H_m^i)*B_{i-s}^n\Bigr)(x)>\frac{\lambda\alpha}{3}}}\\
			&+\abs[\Big]{\set[\Big]{x\in \R^d:\sum\limits_{n=1}^{s/Cs_w}\sum\limits_{m=1}^n\sum\limits_{r=0}^{\Delta s2^{n+2}-1}M\Bigl(\sum\limits_{i\equiv r(mod\;\Delta s2^{n+2})}\Gamma_{m,s}^i*B_{i-s}^n\Bigr)(x)>\frac{\lambda\alpha}{3}}}\\
			=&\mathcal{E}_1+\mathcal{E}_2+\mathcal{E}_3.
		\end{align*}
		We now choose $C$ to satisfy $C\delta>100$. By Chebyshev's inequality and $\norm{v_i}_1\lesssim 1$, we get
		\begin{align*}
			\mathcal{E}_1&\lesssim\frac{1}{\lambda\alpha}\sum\limits_{n=1}^{s/Cs_w}n\Delta s2^n2^{-\delta_2 s}\sum\limits_{Q\in\mathcal{F}}\norm{b_{_Q}^n}_1\\
			&\lesssim\frac{1}{\lambda\alpha}s^2 2^{-s(\delta_2-\frac{1}{Cs_w})}\sum\limits_{n=1}^{\infty}\sum\limits_{Q\in\mathcal{F}}\norm{f_{_Q}^n}_1\\
			&\lesssim\frac{1}{\lambda}2^{-\delta s}\sum\limits_{Q\in\mathcal{F}}\abs{Q}.
		\end{align*}
		The estimate for $\mathcal{E}_2$ follows from positivity and weak $(1,1)$ boundedness of $M$ and inequality \eqref{Seeger2}. Indeed we have
		\begin{align*}
			\mathcal{E}_2&\lesssim\frac{1}{\lambda\alpha}\sum\limits_{n=1}^{s/Cs_w}\sum\limits_{m=1}^n\sum\limits_{r=0}^{\Delta s2^{n+2}-1}\norm[\Big]{\sum\limits_{i\equiv r(mod\;\Delta s2^{n+2})}(\Gamma_{m,s}^i-H_m^i)*B_{i-s}^n}_1\\
			&\lesssim\frac{1}{\lambda\alpha}\sum\limits_{n=1}^{s/Cs_w}n\Delta s2^n2^{-\delta_1 s}\sum\limits_{Q\in\mathcal{F}}\norm{b_{_Q}^n}_1\\
			&\lesssim\frac{1}{\lambda}2^{-\delta s}\sum\limits_{Q\in\mathcal{F}}\abs{Q}.
		\end{align*}
		To estimate $\mathcal{E}_3$, we repeatedly use Cauchy-Schwartz inequality, $L^2$ boundedness of $M$ and \eqref{Seeger1} to get,
		\begin{align*}
			\mathcal{E}_3&\lesssim\frac{1}{\lambda^2\alpha^2}\sum\limits_{n=1}^{s/Cs_w}\frac{s}{s_w}\sum\limits_{m=1}^n n\sum\limits_{r=0}^{\Delta s2^{n+2}-1}\Delta s2^n\norm[\Big]{\sum\limits_{i\equiv r(mod\;\Delta s2^{n+2})}\Gamma_{m,s}^i*B_{i-s}^n}_2^2\\
			&\lesssim\frac{1}{\lambda^2\alpha}\sum\limits_{n=1}^{s/Cs_w}\frac{\Delta^2}{s_w}n^2 s^3 2^{2n}2^{-2\delta_1 s}\sum\limits_{Q\in\mathcal{F}}\norm{b_{_Q}^n}_1\\
			&\lesssim\frac{1}{\lambda^2}2^{-\delta s}\sum\limits_{Q\in\mathcal{F}}\abs{Q}.
		\end{align*}
		Thus the first inequality follows. For the weighted inequality, we apply Chebyshev's inequality to get,
		\begin{align*}
			v(E_\lambda^s)&\lesssim\frac{1}{\lambda\alpha}\sum\limits_{n=1}^{\infty}\sum\limits_{i\in\Z}\sum\limits_{Q\in\mathcal{F}:l(Q)=2^{i-s}}\int\int\abs{(K_n^i-K_0^i)(x-y)}\abs{b_{_Q}^n(y)}\;dy\;v(x)dx\\
			&\lesssim\frac{1}{\lambda\alpha}\sum\limits_{n=1}^{\infty}\sum\limits_{i\in\Z}\sum\limits_{Q\in\mathcal{F}:l(Q)=2^{i-s}}\int\abs{b_{_Q}^n(y)}\frac{1}{2^{id}}\int_{\abs{x-y}\leq 2^{i+2}}v(x)dx\;dy\\
			&\lesssim\frac{1}{\lambda\alpha}\sum\limits_{n=1}^{\infty}\sum\limits_{i\in\Z}\sum\limits_{Q\in\mathcal{F}:l(Q)=2^{i-s}}\int\abs{b_{_Q}^n(y)}\frac{1}{2^{id}}\inf\limits_{z\in Q}\int_{\abs{x-z}\lesssim 2^i}v(x)dx\;dy\\
			&\lesssim\frac{1}{\lambda\alpha}\sum\limits_{n=1}^{\infty}\sum\limits_{i\in\Z}\sum\limits_{Q\in\mathcal{F}:l(Q)=2^{i-s}}\norm{b_{_Q}^n}_{L^1}\inf\limits_{z\in Q} Mv(z)\\
			&\lesssim\frac{1}{\lambda\alpha}\sum\limits_{n=1}^{\infty}\sum\limits_{Q\in\mathcal{F}}\norm{f_{_Q}^n}_{L^1}\inf\limits_{Q} Mv\\
			&\lesssim\frac{1}{\lambda}\sum\limits_{Q\in\mathcal{F}}\abs{Q}\inf\limits_Q Mv.
		\end{align*}	
	\end{proof}
	We now combine the estimates in \Cref{w_est1} into a single weighted estimate using an interpolation scheme given in \cite{FS1}. The proof follows similarly to that of estimate (5.9) in \cite{FS1}. We include a sketch of the proof for the convenience of the reader.
	\begin{lemma}
		\label{w_est2}
		Let $\theta_w=r_w^{-1}$, where $r_w$ as in \Cref{Rhi}. Then we have the following
		\[w(E_\lambda^s)\lesssim\frac{1}{\alpha\lambda^2} 2^{-s\delta(1-\theta_w)}\norm{f}_{L^1(M_{r_w}w)}.\]
	\end{lemma}
	\begin{proof}
		We first note that using \Cref{w_est1} and Lemma 6 of \cite{FS1}, we get
		\[\int_{E_\lambda^s}\min(v(x),t)\;dx\lesssim\frac{1}{\lambda^2}\sum\limits_{Q\in\mathcal{F}}\abs{Q}\min(t2^{-\delta s},\inf\limits_Q Mv),\;\;t\in\R^+.\]
		As $v(x)^{\theta_w}=\theta_w(1-\theta_w)\int_0^\infty \min(v(x),t)t^{-2+\theta_w}\;dt$, we have
		\begin{align*}			
			\int_{E_\lambda^s}v(x)^{\theta_w}\;dx&=\theta_w(1-\theta_w)\int_{E_\lambda^s}\int_0^\infty \min(v(x),t)t^{-2+\theta_w}\;dt\;dx\\
			&\lesssim\frac{1}{\lambda^2}\theta_w(1-\theta_w)\sum\limits_{Q\in\mathcal{F}}\abs{Q}\int_0^\infty\min(t2^{-\delta s},\inf\limits_Q Mv)t^{-2+\theta_w}\;dt\\
			&\lesssim\frac{1}{\alpha\lambda^2} 2^{-s\delta (1-\theta_w)}\norm{f}_{L^1((Mv)^{\theta_w})}.
		\end{align*}
		Choosing $v=w^{r_w}$ produces the desired inequality.
	\end{proof}
	Now that we have all the ingredients, we complete the proof of \Cref{Mainw}. 
	We choose
	\[p_w=1+\frac{1}{\log([w]_{A_{\infty}}+1)}\]
	and $r_w$ as in \Cref{Rhi}. To estimate $\mathcal{H}_{k,1}$, we use Chebyshev's inequality, \Cref{FSi1}, and \Cref{Rhi} to get,
	\begin{align*}
		w\Bigl\{x\in E^c:\sup\limits_k\;\abs{\mathcal{H}_{k,1}(x)}>\frac{\alpha}{5}\Bigr\}&\lesssim\frac{1}{\alpha^{p_w}}\;p_w^{2p_w}p_w'(r_w')^{p_w+\frac{p_w}{p_w'}}\;\norm{g}_{L^p_w(M_{r_w}w)}^{p_w}\\
		&\lesssim\frac{1}{\alpha}\;p_w^{2p_w}p_w'(r_w')^{2p_w-1}\;\norm{f}_{L^1(M_{r_w}w)}\\
		&\lesssim\frac{1}{\alpha}\log([w]_{A_{\infty}}+1)[w]_{A_\infty}[w]_{A_1}\norm{f}_{L^1(w)}.
	\end{align*}
	The estimate for $\mathcal{H}_{k,2}$ follows from \Cref{FSi2} and \Cref{Rhi}, Indeed
	\begin{align*}
		&w\Bigl\{x\in E^c:\sup\limits_k\;\abs{\mathcal{H}_{k,2}(x)}>\frac{\alpha}{5}\Bigr\}\\
		&\lesssim\frac{1}{\alpha^{p_w}}\Bigl(\sum\limits_{n=1}^{\infty}\sum\limits_{m=n+1}^{\infty}\norm{T_m^*(f^n)}_{L^{p_w}(w)}\Bigr)^{p_w}\\
		&\lesssim\frac{1}{\alpha^{p_w}}\;p_w^{2p_w}p_w'(r_w')^{2p_w-1}\;\Bigl(\sum\limits_{n=1}^{\infty}\sum\limits_{m=n+1}^{\infty}2^{-c_22^{m-1}}\Bigl(\int\abs{f^n}^{p_w}M_{r_w}(w)\Bigr)^{\frac{1}{p_w}}\Bigr)^{p_w}\\
		&\lesssim\frac{1}{\alpha}\;p_w^{2p_w}p_w'(r_w')^{2p_w-1}\;\Bigl(\sum\limits_{n=1}^{\infty}2^{-2^n(c_2-c_1(1-\frac{1}{p_w}))}\Bigl(\int\abs{f^n}M_{r_w}(w)\Bigr)^{\frac{1}{p_w}}\Bigr)^{p_w}\\		&\lesssim\frac{1}{\alpha}\;p_w^{2p_w}p_w'(r_w')^{2p_w-1}\;\norm{f}_{L^1(M_{r_w}w)}\Bigl(\sum\limits_{n=1}^{\infty}2^{-\frac{2^n}{p_w}}\Bigr)^{p_w}\\
		&\lesssim\frac{1}{\alpha}\log([w]_{A_{\infty}}+1)[w]_{A_\infty}[w]_{A_1}\norm{f}_{L^1(w)},
	\end{align*}
	where we used the fact that $c_1<c_2$ and $1<p_w<10$. The estimate for $\mathcal{H}_{k,3}$ follows from \Cref{FSi2} and \Cref{Rhi}. Indeed we have
	\begin{align*}
		&w\Bigl\{x\in E^c:\sup\limits_k\;\abs{\mathcal{H}_{k,3}(x)}>\frac{\alpha}{5}\Bigr\}\\
		&\lesssim\frac{1}{\alpha^{p_w}}\Bigl(\sum\limits_{m=1}^{\infty}\norm[\Big]{T_m^*\bigl(\sum\limits_{n=m}^{\infty}g^n\bigr)}_{L^{p_w}(w,E^c)}\Bigr)^{p_w}\\
		&\lesssim\frac{1}{\alpha^{p_w}}\;p_w^{2p_w}p_w'(r_w')^{2p_w-1}\;\Bigl(\sum\limits_{m=1}^{\infty}2^{-c_22^{m-1}}\Bigl(\int\abs[\Big]{\sum\limits_{n=m}^{\infty}g^n(x)}^{p_w}M_{r_w}(w\chi_{E^c})(x)\;dx\Bigr)^{\frac{1}{p_w}}\Bigr)^{p_w}\\
		&\lesssim\frac{1}{\alpha}\;p_w^{2p_w}p_w'(r_w')^{2p_w-1}\;\Bigl(\sum\limits_{m=1}^{\infty}2^{-c_2 2^{m-1}}\Bigl(\int\sum\limits_{Q\in\mathcal{F}}\sum\limits_{n=1}^{\infty}\abs{g_{_Q}^n(x)}M_{r_w}(w\chi_{E^c})(x)\;dx\Bigr)^{\frac{1}{p_w}}\Bigr)^{p_w}\\
		&\lesssim\frac{1}{\alpha}\;p_w^{2p_w}p_w'(r_w')^{2p_w-1}\;\Bigl(\int\sum\limits_{Q\in\mathcal{F}}\sum\limits_{n=1}^{\infty}\abs{g_{_Q}^n(x)}M_{r_w}(w\chi_{E^c})(x)\;dx\Bigr)\\
		&\lesssim\frac{1}{\alpha}\;p_w^{2p_w}p_w'(r_w')^{2p_w-1}\;\Bigl(\sum\limits_{Q\in\mathcal{F}}\int_Q\abs{f(y)}\frac{1}{\abs{Q}}\int_Q M_{r_w}(w\chi_{E^c})(x)\;dx\;dy\Bigr)\\
		&\lesssim\frac{1}{\alpha}\;p_w^{2p_w}p_w'(r_w')^{2p_w-1}\;\Bigl(\sum\limits_{Q\in\mathcal{F}}\int_Q\abs{f(y)}M_{r_w}(w\chi_{\R^d\setminus(100d)Q})(y)\;dy\Bigr)\\
		&\lesssim\frac{1}{\alpha}\;p_w^{2p_w}p_w'(r_w')^{2p_w-1}\;\norm{f}_{L^1(M_{r_w}w)}\\
		&\lesssim\frac{1}{\alpha}\log([w]_{A_{\infty}}+1)[w]_{A_\infty}[w]_{A_1}\norm{f}_{L^1(w)},	
	\end{align*}
	where we used the fact that $M_{r_w}(w\chi_{\R^d\setminus(100d)Q})(x)\lesssim M_{r_w}(w\chi_{\R^d\setminus(100d)Q})(y)$ for $x,y\in Q$ in the third to last step. The proof of this fact follows exactly as that of $M$  (see page 159 of \cite{GR}).
	
	The estimate for $\mathcal{H}_{k,4}$ follows from inequalities \eqref{CZ+M}, \eqref{CZw} and $\norm{\sum\limits_{n=1}^{\infty}f^n}_{L^1(w)}\lesssim\norm{f}_{L^1(w)}$.
	We now turn to the estimate of $\mathcal{H}_{k,b}$. Using the support properties of the kernel \eqref{support}, we get
	\begin{align*}
		w\Bigl\{x\in E^c:\sup\limits_k\;\abs{\mathcal{H}_{k,b}(x)}>\frac{\alpha}{5}\Bigr\}\leq &w\Bigl\{x\in E^c:\sup\limits_k\;\abs[\Big]{\sum\limits_{i>k}\sum\limits_{n=1}^{\infty}\sum\limits_{s=3}^{\infty}(K_n^i-K_0^i)*B_{i-s}^n(x)}>\frac{\alpha}{5}\Bigr\}\\
		\leq&w\Bigl\{x\in E^c:\sup\limits_k\;\abs[\Big]{\sum\limits_{i>k}\sum\limits_{n=1}^{\infty}\sum\limits_{s=3}^{Cns_w-1}(K_n^i-K_0^i)*B_{i-s}^n(x)}>\frac{\alpha}{10}\Bigr\}\\
		&+w\Bigl\{x\in E^c:\sup\limits_k\;\abs[\Big]{\sum\limits_{i>k}\sum\limits_{n=1}^{\infty}\sum\limits_{s=Cns_w}^{\infty}(K_n^i-K_0^i)*B_{i-s}^n(x)}>\frac{\alpha}{10}\Bigr\},
	\end{align*}
	Using Chebyshev's inequality, the first term is dominated by
	\begin{align*}
		&\frac{1}{\alpha}\sum\limits_{n=1}^{\infty}\sum\limits_{i\in\Z}\sum\limits_{s=3}^{Cns_w-1}\sum\limits_{Q\in\mathcal{F}:l(Q)=2^{i-s}}\int\int\abs{(K_n^i-K_0^i)(x-y)}\abs{b_{_Q}^n(y)}\;dy\;w(x)dx\\
		&\lesssim\frac{1}{\alpha}\sum\limits_{n=1}^{\infty}\sum\limits_{i\in\Z}\sum\limits_{s=3}^{Cns_w-1}\sum\limits_{Q\in\mathcal{F}:l(Q)=2^{i-s}}\int\abs{b_{_Q}^n(y)}\frac{1}{2^{id}}\int_{\abs{x-y}\leq 2^{i+2}}w(x)dx\;dy\\
		&\lesssim\frac{1}{\alpha}\sum\limits_{n=1}^{\infty}\sum\limits_{i\in\Z}\sum\limits_{s=3}^{Cns_w-1}\sum\limits_{Q\in\mathcal{F}:l(Q)=2^{i-s}}\int\abs{b_{_Q}^n(y)}\frac{1}{2^{id}}\inf\limits_{z\in Q}\int_{\abs{x-z}\lesssim 2^i}w(x)dx\;dy\\
		&\lesssim\frac{1}{\alpha}\sum\limits_{n=1}^{\infty}\sum\limits_{i\in\Z}\sum\limits_{s=3}^{Cns_w-1}\sum\limits_{Q\in\mathcal{F}:l(Q)=2^{i-s}}\norm{b_{_Q}^n}_{L^1}\inf\limits_Q Mw\\
		&\lesssim\frac{s_w}{\alpha}\sum\limits_{n=1}^{\infty}n\sum\limits_{Q\in\mathcal{F}}\norm{b_{_Q}^n}_{L^1}\inf\limits_Q Mw\\
		&\lesssim\frac{s_w}{\alpha}\sum\limits_{Q\in\mathcal{F}}\int \abs{f_{_Q}(x)}\log\log\left(e^2+\frac{\abs{f_{_Q}(x)}}{\alpha}\right)\;dx\inf\limits_Q Mw\\
		&\lesssim\frac{1}{\alpha}[w]_{A_1}[w]_{A_\infty}\log([w]_{A_\infty}+1)\int \abs{f(x)}\log\log\left(e^2+\frac{\abs{f(x)}}{\alpha}\right)\;w(x)dx.
	\end{align*}
	Let $c>0$ be a constant such that $c\delta(1-\theta_w)\sum\limits_{s=1}^{\infty}2^{-\delta s(1-\theta_w)/3}=1$ and
	\[\lambda_s=\frac{c\delta(1-\theta_w)}{10}2^{-\delta (s-Cs_w)(1-\theta_w)/3}.\]
	Applying \Cref{w_est2} to the weight $w$ with $\lambda=\lambda_s$, we get
	\begin{align*}
		&w\Bigl\{x\in \R^d:\sup\limits_k\;\abs[\Big]{\sum\limits_{i>k}\sum\limits_{n=1}^{\infty}\sum\limits_{s=Cns_w}^{\infty}(K_n^i-K_0^i)*B_{i-s}^n(x)}>\frac{\alpha}{10}\Bigr\}\\
		&\leq w\Bigl\{x\in \R^d:\sum\limits_{s=Cs_w}^{\infty}\sum\limits_{n=1}^{s/Cs_w}\sup\limits_k\;\abs[\Big]{\sum\limits_{i>k}(K_n^i-K_0^i)*B_{i-s}^n(x)}>\sum\limits_{s=1}^{\infty}\frac{c\delta(1-\theta_w)}{10}2^{-\delta s(1-\theta_w)/3}\alpha\Bigr\}\\
		&\leq \sum\limits_{s=Cs_w}^{\infty}w(E_{\lambda_s}^s)\\
		&\lesssim\frac{1}{\alpha}\sum\limits_{s=Cs_w}^{\infty}\frac{1}{(1-\theta_w)^2}2^{2\delta (s-Cs_w)(1-\theta_w)/3}2^{-s\delta(1-\theta_w)}\norm{f}_{L^1(M_{r_w}w)}\\
		&\lesssim\frac{2^{-s_wC\delta(1-\theta_w)}}{\alpha(1-\theta_w)^2}\norm{f}_{L^1(M_{r_w}w)}\sum\limits_{s=Cs_w}^{\infty}2^{-\delta (s-Cs_w)(1-\theta_w)/3}\\
		&\lesssim\frac{2^{-s_wC\delta(1-\theta_w)}}{\alpha(1-\theta_w)^3}\norm{f}_{L^1(M_{r_w}w)}\\
		&\lesssim\frac{1}{\alpha}[w]_{A_1}\norm{f}_{L^1(w)}.
	\end{align*}
	Hence the proof of \Cref{Mainw} concludes.

\end{document}